\documentclass{article}
\usepackage[utf8]{inputenc}
\usepackage{amsmath,amssymb,amsthm}
\usepackage{hyperref,graphicx,color}
\usepackage{enumerate}
\usepackage{a4wide}

\usepackage{todonotes}

\newtheorem{thm}{Theorem}
\newtheorem{lem}[thm]{Lemma}
\newtheorem{prop}[thm]{Proposition}
\newtheorem{cor}[thm]{Corollary}
\newtheorem{rem}[thm]{Remark}
\newcommand\eps{\ensuremath{\varepsilon}}

\title{Corrigendum to \emph{Ramification of multiple eigenvalues for the 
Dirichlet-Laplacian in perforated domains} published in Journal of Functional Analysis 283 (2022) 109718}
\author{Laura Abatangelo, Corentin Léna, Paolo Musolino}
\date{May 7, 2026}

\begin{document}

\maketitle

\begin{abstract}
    We fix the proof of \cite[Theorem 1.17]{JFA2022} which does not in general gives the correct coefficients in the asymptotic behavior of eigenvalues for the Dirichlet Laplacian when a small hole is removed from the domain.
\end{abstract}

\section{Introduction}
\label{sec:intro}

Let us recall the setting of \cite{JFA2022}. We fix an open, bounded and simply connected set $\Omega\subset \mathbb R^2$, of class $C^{1,\alpha}$ for some $\alpha>0$. We consider the eigenvalues $\{\lambda_j\}_{j\ge 1}$ of the Dirichlet Laplacian in $\Omega$ (counted with multiplicities and arranged in non-decreasing order). We call those \emph{unperturbed}. For our purpose, it is convenient to see them as the eigenvalues of the quadratic form
\begin{equation}
\label{eq:Q}
	Q(u)=\int_{\Omega}\left|\nabla u\right|^2\,dx
\end{equation}
on $H^1_0(\Omega)$, relative to the scalar product 
\begin{equation}
\label{eq:L2}
	\langle u,v\rangle=\int_{\Omega} u\,v\,dx
\end{equation}
on $L^2(\Omega)$ (we are considering real Hilbert spaces throughout). In other words, $u\in H^1_0(\Omega)$ is an eigenfunction of $Q$, associated with the eigenvalue $\lambda$, when
\begin{equation*}
 Q(u,\varphi)=\lambda\,\langle u,\varphi\rangle\quad\text{for all } \varphi\in H^1_0(\Omega)
\end{equation*} 
(everywhere in this note, with a slight abuse of notation, we use the same letter to denote a quadratic form and its associated symmetric bilinear form).

We fix another set $\omega\subset\mathbb R^2$ satisfying the same hypotheses as $\Omega$ and assume that $0\in\Omega$ and $0\in\omega$. Then, there exists $\eps_0>0$ such that $\eps\overline{\omega}\subset\Omega$ for all $\eps\in (0,\eps_0)$. For such an $\eps$, we define the \emph{perforated domain} $\Omega_\eps:=\Omega\setminus\eps\overline{\omega}$. We now consider the eigenvalues $\{\lambda_{j}^\eps\}_{j\ge1}$ of the Dirichlet Laplacian in $\Omega_\eps$, called \emph{perturbed}. They can be seen as the eigenvalues of the quadratic form  defined by Equation  \eqref{eq:Q} on the domain $H^1_0(\Omega_\eps)$, relative to the scalar product defined by Equation \eqref{eq:L2} on $L^2(\Omega_\eps)$. With a slight abuse of notation, we write $Q$ and $\langle\cdot\,,\cdot\rangle$ for all these quadratic forms and scalar products, although they act on different spaces. In addition, we write $\|f\|$ for the $L^2$-norm of a function $f$, whatever its domain of definition.

Let us note that any function defined on $\Omega_\eps$ can be extended by $0$ on $\eps \overline{\omega}$ to obtain a function defined on $\Omega$. This induces an isometric injection of $H^1_0(\Omega_\eps)$ (respectively $L^2(\Omega_\eps)$) into $H^1_0(\Omega)$ (respectively $L^2(\Omega)$), which allows us to identify the former space with a subspace of the latter. It then follows from the minmax principle that $\lambda_j\le\lambda_j^\eps$ for all $\eps\in(0,\eps_0)$ and $j\in\mathbb N$. 

In the rest of this note, we fix an unperturbed eigenvalue   $\lambda_N$ of multiplicity $m$ (with $N$ chosen so that $\lambda_{N-1}<\lambda_N$) and we denote by $E$ the associated eigenspace, which is by assumption an $m$-dimensional  subspace of $H_0^1(\Omega)$. We study various limits and asymptotic behaviors as $\eps\to0$. It is known that 
 \begin{equation}
 \label{eq:conv-ev}
 	\lambda_j^\eps\to \lambda_j 
 \end{equation}  
for all $j\in\mathbb N$ (see for instance \cite[Proposition 1.3]{JFA2022}) so that, in particular,
 \begin{equation*}
 	\lambda_{j+N-1}^\eps\to \lambda_N\quad \text{for all }1\le j\le m. 
 \end{equation*} 
We define the \emph{spectral shifts} by $\nu_j^\eps:=\lambda_{j}^\eps-\lambda_N$ for all $j\in\mathbb N\setminus\{0\}$. We want to understand the behavior of the vanishing spectral shifts, i.e. those for which $N\le j\le N+m-1$.  Our goal is the following theorem (which corresponds to Theorem 1.17 in \cite{JFA2022}).

\begin{thm}\label{thm:main}
	  The following alternative holds.
	\begin{enumerate}[(i)]
		\item \label{thm:case-i} For all $u\in E$, $u(0)=0$. Then, there exist a finite, non-increasing, sequence of positive integers $\{\ell_j\}_{j=1}^m$ and a finite sequence of positive numbers $\{\hat\nu_j\}_{j=1}^m$ such that 
		\begin{equation}\label{eq:case-i}
		\nu^\eps_{j+N-1}=\hat\nu_j\eps^{2\ell_j}+o\left(\eps^{2\ell_j}\right)\quad \text{for all }1\le j\le m.
		\end{equation}
		\item \label{thm:case-ii} There exists $u\in E$ such that $u(0)\neq0$. Then, there exist a finite, non-increasing, sequence of positive integers $\{\ell_j\}_{j=1}^{m-1}$, a finite sequence of positive numbers $\{\hat\nu_j\}_{j=1}^{m-1}$and a positive constant $c$ such that 
	\begin{align}		
		\label{eq:case-ii-1} \nu^\eps_{j+N-1}&=\hat\nu_j\eps^{2\ell_j}+o\left(\eps^{2\ell_j}\right)\quad \text{for all }1\le j\le m-1,\\
		\label{eq:case-ii-2} \nu^\eps_{m+N-1}&=\frac{c}{|\log(\eps)|}+o\left(\frac{1}{|\log(\eps)|}\right).
	\end{align}
	\end{enumerate} 
\end{thm}

In \cite[Section 4]{JFA2022} we attempted to prove the theorem by an iterative procedure relying on the so-called \emph{Lemma on Small Eigenvalues} due to Yves Colin de Verdière (see \cite{Colin1986,ColboisColin1987}). However, this method is incorrect since the key estimate on page 38
\[
|q^\eps_{p-1}(v,w)| \leq o\left( \Big(\frac{\rho^\eps_{k_p-1}}{\rho^\eps_{k_p}}\Big)^{\frac12} \right) \, \|v\|\,\|w\|\]
does not hold in general. In this note, we carry out the proof using a different approach. We reduce the study of the spectral shifts to a finite dimensional eigenvalue problem and find the asymptotic behavior of its eigenvalues by combining an estimate inspired by the aforementioned lemma with standard tools in linear algebra. Our analysis shows that the coefficients in the asymptotic expansion provided by \cite[Theorem 1.17]{JFA2022} (corresponding to the $\{\hat \nu_j\}_{j=1}^m$ in Theorem \ref{thm:main}) are in general wrong, and gives the correct coefficients. 

\section{Reduction to a finite dimensional problem}
\label{sec:finite}

We have seen that we can isolate a group of $m$ eigenvalues in $\{\lambda_j^\eps\}_{j\ge1}$ which converge to $\lambda_N$. In order to state this fact more explicitly, we define a \emph{spectral gap}
 \[\gamma :=\frac12 \min\{ \lambda_N - \lambda_{N-1}, \lambda_{N+m} - \lambda_{N}  \}>0\]
and we observe that \eqref{eq:conv-ev} implies that 
 \begin{equation}
 	\nu_{j+N-1}^\eps\to 0\quad\text{for all } 1\le j\le m
 \end{equation}
and that there exists $\eps_1>0$ (with $\eps_1\le\eps_0$) such that, for all $0<\eps\le\eps_1$,
\begin{equation}
\label{eq:gap-condition}
 \nu_j^\eps\le-\gamma\text{ if }j\le N-1\quad\text{and}\quad\nu_j^\eps\ge
 \gamma\text{ if }j\ge N+1.
\end{equation}
 For such an $\eps$, we define the $m$-dimensional subspace $E_\eps$ of $L^2(\Omega_\eps)$ by
\begin{equation*}
	E_\eps:=E(\lambda_{N}^\eps)+E(\lambda_{N+1}^\eps)+\cdots+E(\lambda_{N+m-1}^\eps)
\end{equation*}  
(where $E(\lambda_j^\eps)$ denotes the eigenspace associated with $\lambda_j^\eps$). We note that $E_\eps$ is a subspace of $H_0^1(\Omega_\eps)$.

In order to study the spectral shifts, we define the quadratic form $q$ on $H^1_0(\Omega)$ by
\begin{equation*}
	q(u):=Q(u)-\lambda_N\|u\|^2=\int_\Omega \left|\nabla u\right|^2\,dx-\lambda_N\int_\Omega u^2\,dx
\end{equation*}
and we observe that 
\begin{equation}
\label{eq:evE}
 \forall\,u\in E,\ \forall\,\varphi\in H^1_0(\Omega),\ q(u,\varphi)=0.
\end{equation} 
We now define $q_\eps$ as the restriction of $q$ to the space $H^1_0(\Omega_\eps)$. A characterization of the spectral shifts follows immediately.
\begin{lem}\label{lem:shifts-1} The eigenvalues of $q_\eps$
 relative to the scalar product in $L^2(\Omega_\eps)$ are $\{\nu_j^\eps\}_{j\ge1}$, with the same associated eigenspaces as $\{\lambda_j^\eps\}_{j\ge1}$. In particular, the spectral shifts $\{\nu_{j+N-1}^\eps\}_{j=1}^m$ are the eigenvalues of the restriction of $q_\eps$ to $E_\eps$ relative to the restriction of the scalar product.
\end{lem}

We now seek a reformulation of this characterization in the fixed $m$-dimensional vector space $E$.  We first define the $L^2$-orthogonal projector 
\begin{equation*}
	\Pi_\eps:L^2(\Omega_\eps)\to E_\eps.
\end{equation*}
Then, following \cite{JST2019}, we define the linear map
\[P_\eps: H^1_0(\Omega)\to H^1_0(\Omega_\eps), \quad P_\eps(u):= u - V_\eps^u\]
in such a way that $P_\eps$ is a $Q$-orthogonal projection. In other words, for a given $u\in H^1_0(\Omega)$, the $u$-capacitary potential $V_\eps^u$ (see \cite{JST2019} for more details and a justification of the terminology) is defined by the properties
\begin{equation}\label{eq:orthoV}
u-V_\eps^u\in H^1_0(\Omega_\eps)\quad\mbox{ and }\quad \ \forall\,\varphi\in H^1_0(\Omega_\eps),\ Q(V_\eps^u,\varphi)=0.
\end{equation}
The next result follows from Proposition B.1 and Lemma A.1 in \cite{JST2019}.
\begin{lem}\label{l:V} For all $u\in H^1_0(\Omega)$, $Q\left(V_\eps^u\right)\to0$ and $\|V_\eps^u\|^2=o\left(Q\left(V_\eps^u\right)\right)$.
\end{lem} 

We now define the mapping 
\begin{equation*}
	T_\eps:E\to E_\eps, \quad T_\eps(u)=\Pi_\eps(P_\eps(u)).
\end{equation*}
For any $u\in E$, we introduce the notation $h_\eps^u=P_\eps(u)-T_\eps(u)$, so that 
\begin{equation*}
	T_\eps(u)=u-V_\eps^u-h_\eps^u.
\end{equation*}
Let us note that, since $\Pi_\eps$ is a spectral projector,  $h_\eps^u$ is both $L^2$-orthogonal and $Q$-orthogonal to the subspace $E_\eps$.

The following estimate is the crucial step in our proof of Theorem \ref{thm:main}. It is inspired by the \emph{Lemma on Small Eigenvalues} (see \cite{Colin1986,ColboisColin1987,Courtois1995}).
\begin{lem}\label{l:lpvp} There exists a constant $C_0$ such that, for all $0<\eps\le\eps_1$,  
\begin{align}
	\label{eq:lsev} \forall\,u\in E,\ \left\|h_\eps^u\right\|&\le C_0\left\|V_\eps^u\right\|\,.
\end{align}
\end{lem}  
\begin{proof} For all $\eps>0$, we denote by $\{u_j^\eps\}_{j\ge 1}$ an orthonormal basis of $L^2(\Omega_\eps)$ consisting of eigenfunctions associated with the perturbed eigenvalues $\{\lambda_j^\eps\}_{j\ge 1}$. Let us now fix $u\in E$. The function $h_\eps^u$ is orthogonal to the subspace $E_\eps$ of $H^1_0(\Omega_\eps)$ for the bilinear form $q_\eps$. Thus,
\[
h_\eps^u = \sum_{j=1}^{N-1} c_j u_j^\eps + \sum_{j=N+m}^{+\infty} c_j u_j^\eps=: h_\eps^{u-} + h_\eps^{u+}
\]
and we have
\begin{align*}
 \| h_\eps^u \|^2 = \| h_\eps^{u+} \|^2 + \| h_\eps^{u-} \|^2 \quad \text{and}\quad
q_\eps( h_\eps^u ) = q_\eps(h_\eps^{u+}) + q_\eps( h_\eps^{u-}).
\end{align*}

Using the eigenvalues $\{\nu_j^\eps\}_{j\ge1}$ of $q_\eps$, we can write
\begin{align*}
    & \|h_\eps^{u-}\|^2 = \sum_{j=1}^{N-1} c_j^2 \quad \text{and}\quad q_\eps(h_\eps^{u-}) = \sum_{j=1}^{N-1} \nu_j^\eps c_j^2;\\
    & \|h_\eps^{u+}\|^2 = \sum_{j=N+m}^{+\infty} c_j^2 \quad \text{and}\quad  q_\eps(h_\eps^{u+}) = \sum_{j=N+m}^{+\infty} \nu_j^\eps c_j^2.
\end{align*}
Since $\eps\le \eps_1$, we have $\nu_j^\eps\le-\gamma$ for $j\le N-1$ and $\nu_j^\eps\ge \gamma$ for $j\ge N+1$. Therefore, 
\begin{equation}\label{eq:qe1}
|q_\eps(h_\eps^{u\pm})| \geq \gamma \|h_\eps^{u\pm}\|^2.   
\end{equation}
We can compute explicitly the quadratic form $q_\eps$ on $h_\eps^{u\pm}$ using \eqref{eq:orthoV}, \eqref{eq:evE} and mutual orthogonality:
\begin{align*}
    q_\eps(h_\eps^\pm) &= q_\eps(h_\eps^{u}, h_\eps^{u\pm}) = q_\eps(u-V_\eps^u, h_\eps^{u\pm})=q(u,h_\eps^{u\pm})-q(V_\eps^u,h_\eps^{u\pm})\\
    &=-Q(V_\eps^u,h_\eps^{u\pm})+\lambda_N\langle V_\eps^u,h_\eps^{u\pm}\rangle=\lambda_N\langle V_\eps^u,h_\eps^{u\pm}\rangle.
\end{align*}
Together with \eqref{eq:qe1}, this implies successively 
\begin{align*}
\gamma \|h_\eps^{u\pm}\|^2 &\leq |q_\eps(h_\eps^{u\pm})| \leq \lambda_N \|V_\eps^{u}\| \|h_\eps^{u\pm}\|,\\
\|h_\eps^{u\pm}\| & \leq \frac{\lambda_N}{\gamma} \|V_\eps^{u}\|,\\
 \|h_\eps^{u}\|&=\sqrt{\|h_\eps^{u-}\|^2+\|h_\eps^{u+}\|^2} \leq \frac{\sqrt2 \lambda_N}{\gamma} \|V_\eps^{u}\|.\qedhere
\end{align*}
\end{proof}

It follows immediately from Lemmas \ref{l:V} and \ref{l:lpvp} that, given any $u\in E$, $V_\eps^u\to 0$ and $h_\eps^u\to 0$ in $L^2(\Omega)$, and therefore $T_\eps(u)\to u$. Since $E$ is finite dimensional, this implies
\[\sup_{u\in E, \|u\|=1}\|T_\eps u-u\|\to0.\] 
As a consequence, for $\eps>0$ small enough, $T_\eps:E\to E_\eps$ is injective and therefore bijective since, by construction, $E_\eps$ has dimension $m=\mbox{dim}(E)$. 

Assuming bijectivity from now on, we define the quadratic forms on $E$
\[a_\eps(u,v):=q_\eps\left(T_\eps u,T_\eps v\right)
\quad\text{and}\quad b_\eps(u,v):=\langle T_\eps u, T_\eps v\rangle.\]
We then immediately deduce from Lemma \ref{lem:shifts-1} another characterization of the $m$ vanishing spectral shifts.
\begin{lem}\label{lem:shifts-2} For $\eps>0$ small enough, the spectral shifts $\{\nu_{j+N-1}^\eps\}_{j=1}^m$ are the eigenvalues of the quadratic form $a_\eps(\cdot\,,\cdot)$, relative to the scalar product $b_\eps(\cdot\,,\cdot)$, both defined on $E$.
\end{lem} 

Moreover, the limiting behavior of $a_\eps(\cdot,\cdot)$ and $b_\eps(\cdot,\cdot)$ is described by the following lemma.
\begin{lem}\label{l:qf-cap} 
For any $u$ and $v$ in $E\setminus\{0\}$,
\begin{align}
	\label{eq:asym-a} a_\eps(u,v)&=Q(V_\eps^{u},V_\eps^{v})+O\left(\left\|V_\eps^{u}\right\|\,\left\|V_\eps^{v}\right\|\right),\\
	\label{eq:asym-b} b_\eps(u,v)&=\langle u,v\rangle+o(1).
\end{align} 
\end{lem}
\begin{proof} Since $T_\eps w\to w$ in $L^2(\Omega)$ for all $w\in E$, the estimate \eqref{eq:asym-b} is immediate. 
\begin{align*}
	a_\eps(u,u)&=q_\eps(u-V_\eps^{u}-h_\eps^{u},v-V_\eps^{v}-h_\eps^{v})=q_\eps(u-V_\eps^{u},v-V_\eps^{v}-h^{v}_\eps)\\
	&=-q(V_\eps^{u},v-V_\eps^{v}-h_\eps^{v})=q(V_\eps^{u},V_\eps^{v}+h_\eps^{v})\\
	&=Q(V_\eps^{u},V_\eps^{v})-\lambda_N\langle V_\eps^{u},V_\eps^{v}\rangle+Q(V_\eps^{u},h_\eps^{v})-\lambda_N\langle V_\eps^{u},h_\eps^{v} \rangle\\
	&=Q(V_\eps^{u},V_\eps^{v})-\lambda_N\langle V_\eps^{u},V_\eps^{v}\rangle-\lambda_N\langle V_\eps^{u},h_\eps^{v} \rangle.
\end{align*} 
Combining this identity with estimate \eqref{eq:lsev}, we find \eqref{eq:asym-a}.
\end{proof} 

\section{Choice of suitable basis}
\label{sec:basis}

To estimate the vanishing spectral shifts, we study the quadratic forms $a_\eps(\cdot\,,\cdot)$ and $b_\eps(\cdot\,,\cdot)$ in a suitable basis. To that end, let us review the asymptotic analysis in \cite[Section 2]{JFA2022}. We first observe that any eigenfunction $u$ associated with an unperturbed eigenvalue is analytic in $\Omega$. In particular, it has a finite (and integer) order of vanishing at $0$, which we denote by $\kappa(u)$. More explicitly, 
\begin{equation}\label{eq:principal-part}
u(x)=u_\sharp(x)+O\left(|x|^{\kappa(u)+1}\right)\quad\text{as }x\to0
\end{equation}
where $u_\sharp$ is a homogeneous polynomial of degree $\kappa(u)\in\mathbb N$. We call $u_\sharp$ the \emph{principal part} of $u$. The fact that $u$ is an eigenfunction implies that $u_\sharp$ is harmonic. In the case $u(0)\neq0$, $\kappa(u)=0$ and $u_\sharp$ is constant equal to $u(0)$.

For any $u\in E\setminus\{0\}$ such that $u(0)=0$, we define the function $\Phi^u:\mathbb R^2\to \mathbb R$ by the properties
\begin{enumerate}[(i)]
\item in $\omega$, $\Phi^u$ equals $u_{\sharp}$,
\item in $\mathbb R^2\setminus \omega$, $\Phi^u$ equals $\mathsf{u}$, the unique function harmonic in $\mathbb R^2\setminus \overline{\omega}$, bounded at $\infty$, continuous in $\mathbb R^2\setminus \omega$, and such that $\mathsf{u}=u_\sharp$ on $\partial\omega$. 
\end{enumerate} 
  
\begin{prop} \label{p:cap-estimate} Let $u$ and $v$ be in $E\setminus\{0\}$.
\begin{enumerate}[(i)]
\item If $u(0)\neq 0$ and $v(0)\neq 0$, then
\begin{equation}\label{eq:cap-non-zero}
Q(V^u_\eps,V^v_{\eps})=\frac{2\pi\,u(0)\,v(0)}{|\log(\eps)|}+o\left(\frac{1}{|\log(\eps)|}\right).
\end{equation}
\item If $u(0)=v(0)=0$, then
\begin{equation}\label{eq:cap-uv}
Q(V^u_\eps,V^v_{\eps})=\eps^{\kappa(u)+\kappa(v)}\int_{\mathbb R^2}\nabla \Phi^u\cdot\nabla \Phi^v\,dx+o\left(\eps^{\kappa(u)+\kappa(v)}\right).
\end{equation}
\item If $u(0)=0$ and $v(0)\neq0$, then
 \begin{equation}\label{eq:cap-mixed}
Q(V^u_\eps,V^v_{\eps})=2\pi\left(\lim_{x\to\infty}\Phi^u(x)\right)v(0)\,\frac{\eps^{\kappa(u)}}{|\log(\eps)|}+o\left(\frac{\eps^{\kappa(u)}}{|\log(\eps)|}\right).
\end{equation}
\end{enumerate}
\end{prop}

\begin{proof} The estimate \eqref{eq:cap-non-zero} follows immediately from \cite[Equation (21)]{JFA2022} and \eqref{eq:cap-uv} is exactly \cite[Remark 2.16]{JFA2022}. 

To prove Estimate \eqref{eq:cap-mixed}, we refer to \cite[Section 2]{JFA2022} in more detail. We first recall that the definition of perforated domains (from Section \ref{sec:intro}) and of the corresponding $u$-capacitary potentials (from Section \ref{sec:finite}) can be extended to a negative $\eps$ close enough to $0$. Indeed, for such an $\epsilon$, we simply set $\Omega_{\eps}=\Omega\setminus |\eps|\,\overline{(-\omega)}$, where $-\omega$ denotes the image of $\omega$ by a symmetry centered at the origin, and we define $V_u^\eps$ accordingly for all $u\in H^1_0(\Omega)$.  In order to recover the notation used in \cite[Section 2]{JFA2022}, we substitute $u^a$, $u^b$ for $u$, $v$; $\overline k^a$, $\overline k^b$ for $\kappa(u)$, $\kappa(v)$ and $\mathrm{Cap}_\Omega(\varepsilon\overline{\omega},u^a,u^b)$ for $Q(V_\eps^u,V_\eps^v)$. According to \cite[Theorem 2.13]{JFA2022},  the following series expansion holds for $\eps\neq0$ in a sufficiently small neighborhood of $0$:
\begin{equation*}
\mathrm{Cap}_\Omega(\varepsilon\overline{\omega},u^a,u^b)=\sum_{n=0}^\infty\varepsilon^n\sum_{l=0}^{n+1} \frac{c_{(n,l)}}{(r_0+(2\pi)^{-1}\log|\varepsilon|)^{l}},
\end{equation*}
where the constant $r_0$ is defined in \cite[Proposition 2.8]{JFA2022} and the coefficients $c_{(n,l)}$ are explicitly constructed in \cite[Theorem 2.13]{JFA2022}. 

One can check that the computations in \cite[Section 2.7]{JFA2022} are also correct in the case where $\overline k^a>0 $ and $\overline k^b=0$, rather than $\overline k^a>0 $ and $\overline k^b>0$. The vanishing of derivatives stated in \cite[Equation (22)]{JFA2022} and the same arguments as in \cite[Section 2.7]{JFA2022} imply that 
\begin{align*}
	& \ c_{(n,0)}=c_{(n,1)}=0 \qquad \forall\, n<\overline k^a,\\
	& \ c_{(n,l)}=0 \qquad \forall\, (n,l)\text{ such that } 2\le l \le n+1 \text{ and } n-l+1<\overline k^a.
\end{align*}   

Moreover, 
\begin{align*}
	c_{(\overline k^a,0)}&=-u^b(0)\,\int_{\partial \omega}  \frac{\partial u^a_{\mathrm{m},\overline{k}^a}}{\partial \nu_{\omega}}\,d\sigma,\\
 	c_{(\overline k^a,1)}&= -\Bigg(\int_{\partial\omega}u^a_{\#,\overline{k}^a}\rho^i_{0}\,d\sigma\Bigg)\Bigg( u^b(0)\,\int_{\partial \omega} \tilde{v}_0
 \, d\sigma\Bigg),
\end{align*}
where $u^a_{\mathrm{m},\overline{k}^a}$ is defined in \cite[Proposition 2.8]{JFA2022}, $u^a_{\#,\overline{k}^a}$ is the principal part at $0$ of $u^a$, {\it i.e.},
\[
u^a_{\#,\overline k^a}(t)\equiv\sum_{\substack{(h,j)\in \mathbb{N}^2\\ h+j=\overline k^a}}\frac{\partial_1^h \partial_2^{j}u^a(0)}{h! j!}t^h_1 t_2^j\qquad\qquad\quad  \forall t\in  \mathbb{R}^2\, ,
\]
(see also Equation \eqref{eq:principal-part}) and $\rho^i_0$ and $\tilde{v}_0$ are respectively defined in \cite[Proposition 2.6]{JFA2022} and \cite[Proposition 2.9]{JFA2022}. We also note that, by \cite[Proposition 2.6]{JFA2022}, we have $\int_{\partial \omega}\rho^i_0\, d\sigma=1$.
As pointed out on page 30 of \cite{JFA2022}, the Divergence Theorem and the decay properties of the radial derivative of $u^a_{\mathrm{m},\overline{k}^a}$ (see \cite[Proposition 2.75]{Fo95}) imply
\[\int_{\partial \omega}  \frac{\partial u^a_{\mathrm{m},\overline{k}^a}}{\partial \nu_{\omega}}\,d\sigma=0.\]
In addition, it is shown on page 31 of \cite{JFA2022}  that $\tilde{v}_0=\rho^i_0$ and in \cite[Proof of Lemma 7.2]{DaMuRo15} that 
\[\int_{\partial\omega}u^a_{\#,\overline{k}^a}\rho^i_{0}\,d\sigma=\lim_{t\to\infty}\mathsf{u}^a_{\overline{k}^a}(t),\]
where $\mathsf{u}^a_{\overline{k}^a}$ is the unique solution to the boundary value problem
\begin{equation*}\label{eq:ulk}
\left\{
\begin{array}{ll}
\Delta \mathsf{u}^a_{\overline{k}^a}=0&\text{in }\mathbb{R}^2\setminus\overline{\omega}\,,\\
\mathsf{u}^a_{\overline{k}^a}=u^{a}_{\#,\overline{k}^a}&\text{on }\partial\omega\,,\\
\mathsf{u}^a_{\overline{k}^a}\text{ is bounded at }\infty\,.
\end{array} 
\right.
\end{equation*}
In the notation used previously, $\mathsf{u}^a_{\overline{k}^a}=\Phi^u$ in $\mathbb R^2\setminus \omega$. We conclude that
\begin{equation*}
	c_{(\overline k^a,0)}=0\quad\text{and}\quad c_{(\overline k^a,1)}=-\left(\lim_{t\to\infty}\mathsf{u}^a_{\overline{k}^a}(t)\right)\,u^b(0).
\end{equation*} 
In particular, this implies
\begin{equation*}
\mathrm{Cap}_\Omega(\varepsilon\overline{\omega},u^a,u^b)=-2\pi\,\left(\lim_{t\to\infty}\mathsf{u}^a_{\overline{k}^a}(t)\right)\,u^b(0)\,\frac{\eps^{\overline k^a}}{\log|\eps|}+o\Bigg(\frac{\eps^{\overline k^a}}{\log|\eps|}\Bigg) \qquad \text{as $\eps\to 0$},
\end{equation*}
which is the desired estimate. \end{proof}
  
  We proceed with the choice of a suitable basis, which relies on the grouping of unperturbed eigenfunctions according to their order of vanishing. As shown in \cite[Appendix A]{JFA2022}, there exists a decomposition of $E$ into an $L^2$-orthogonal sum of non-trivial subspaces
\[E=E_1\oplus \dots \oplus E_p\]
(with $p\ge 1$), associated with a decreasing sequence of integers 
\[k_1>\dots>k_p\ge0,\]
such that, for any $i\in\{1,\dots,p\}$, all non-zero functions in $E_i$ have an order of vanishing at $0$ equal to $k_i$. For future reference, we denote by $m_i$ the dimension of $E_i$, so that 
\[m=m_1+\dots+m_p.\]
In addition, we denote by $\{\ell_j\}_{j=1}^{m}$ the finite sequence of the orders of vanishing $\{k_i\}_{i=1}^p$, counted with the multiplicities $\{m_i\}_{i=1}^p$; that is, we set
\begin{equation}
\label{eq:seq-l}
\begin{aligned}
	\ell_1=\dots=\ell_{m_1}&=k_1,\\
	\ell_{m_1+1}=\dots=\ell_{m_1+m_2}&=k_2,\\
	&\vdots\\
	\ell_{m_1+\dots+m_{p-1}+1}=\dots=\ell_{m}&=k_p.
\end{aligned}
\end{equation}

\begin{rem} \label{rem:mult} As pointed out in \cite[Remark 1.12]{JFA2022}, we always have $m_p=1$ if $k_p=0$ (equivalently, $0$ has multiplicity at most $1$ in the sequence $\{\ell_j\}_{j=1}^m$). Moreover, in the two-dimensional setting we are considering in this note, $m_j\le 2$ for all $1\le j\le p$.\end{rem}

\begin{lem}\label{l:basis}
 There exists an orthonormal basis $\{u_1,\dots,u_m\}$ of $E$ which is adapted to the order decomposition, in the sense that each $u_j$ belongs to some $E_i$.
\end{lem}
Indeed, since each $E_i$ is a finite dimensional subspace of $E$, we can pick an orthonormal basis for it; since the subspaces $E_i$ ($1\le i\le p$) are by assumption mutually orthogonal, the concatenation of these bases gives an orthonormal basis of $E$.    We fix such a basis $\{u_j\}_{j=1}^m$ from now on.

In the case where $k_p>0$ (equivalent to $u_m(0)\neq0)$, we define the $m$-by-$m$ real and symmetric matrix
\begin{equation}\label{eq:G}
    \hat A_0:= \left[\int_{\mathbb R^2}\nabla \Phi^{u_i}\cdot\nabla  \Phi^{u_j}\,dx\right]_{i,j}=[\hat a^{(0)}_{i,j}]_{i,j}.
\end{equation}
To give a geometric interpretation of $\hat A_0$, let us consider the vector space $\mathcal X$, defined as the subspace of $C^0\left(\mathbb R^2\right)$ consisting of the functions $\Phi$ satisfying
\begin{enumerate}[(i)]
	\item $\Phi(0)=0$,
    \item $\Phi$ is harmonic in $\omega$,
    \item $\Phi$ is harmonic in $\mathbb R^2\setminus \overline{\omega}$ and bounded at infinity,
\end{enumerate}
equipped with the bilinear form
\begin{equation}\label{eq:cs}
\langle \Phi,\Psi\rangle_{\mathcal X}:=\int_{\mathbb R^2} \nabla \Phi\cdot\nabla \Psi\,dx.
\end{equation}
Then, $(\mathcal X,\langle\cdot,\cdot\rangle_{\mathcal X})$ is a real inner product space and $\hat A_0$ is the Gram matrix associated with the family $\{\Phi^{u_1},\dots, \Phi^{u_m}\}$.

\begin{lem}\label{l:Ujlinind}
    The functions $\{\Phi^{u_j}\}_{j=1}^m$ are linearly independent.
\end{lem}
\begin{proof}
    Suppose that 
    \[
    c_1 \Phi^{u_1} + \ldots + c_m \Phi^{u_m}=0. 
    \]
    The function on the left-hand side, restricted to $\omega$, can be split into a sum of homogeneous polynomials of respective degrees $k_1,\dots, k_p$. Since each of these polynomials must be $0$, we obtain
    \[
    \begin{cases}    
  c_1 (u_1)_\sharp + \ldots + c_{m_1} (u_{m_1})_\sharp=0,\\
    c_{m_1+1} (u_{m_1+1})_\sharp + \ldots + c_{m_1+m_2} (u_{m_1+m_2})_\sharp=0,\\
    \hspace{50pt}\vdots\\
    c_{m_1+m_2+\dots+m_{p-1}+1} (u_{m_1+m_2+\dots+m_{p-1}+1})_\sharp+ \ldots + c_{m}(u_{m})_\sharp=0.
    \end{cases}
    \]
   The linear combination $c_1 u_1 + \ldots + c_{m_1} u_{m_1}$ is a function belonging to $E_1\subset E$. According to the first equation, it has an order of vanishing at $0$ higher than $k_1$, therefore it must be zero by definition of the space $E_1$. Since $u_1,\ldots,u_{m_1}$ are linearly independent, this implies $c_1=\ldots=c_{m_1}=0$. Applying the same argument to the other lines, we obtain $c_1=\dots=c_m=0$.  
\end{proof}

\begin{cor}
    The matrix $\hat A_0$ defined in \eqref{eq:G} is positive-definite.
\end{cor}

 We now gather together the results of this section. Let $A(\eps)=[a_{i,j}(\eps)]_{1\leq i,j\leq m}$ and $B(\eps)=[b_{i,j}(\eps)]_{1\leq i,j\leq m}$ be, respectively, the matrices of $a_\eps(\cdot\,,\cdot)$ and $b_\eps(\cdot\,,\cdot)$ in the basis $\{u_j\}_{j=1}^m$; that is,
\begin{equation}\label{eq:def-AB}
 a_{i,j}(\eps)=a_\eps(u_i,u_j)\quad\text{and}\quad b_{i,j}(\eps)=b_\eps(u_i,u_j).
\end{equation}
\begin{prop} \label{p:entries} The matrices $A(\eps)$ and $B(\eps)$ are real, symmetric, and satisfy the following asymptotic estimates.
\begin{enumerate}[(i)]
	\item If $u_m(0)=0$ (that is, if $k_p>0$, or equivalently if $u(0)=0$ for all $u\in E$), then
	\begin{equation}
	\label{eq:asym1} a_{i,j}(\eps)=\hat a^{(0)}_{i,j}\,\eps^{\ell_i+\ell_j}+o\left(\eps^{\ell_i+\ell_j}\right) \quad\text{ for all }1\le i,j\le m,
	\end{equation}
where $\{\ell_j\}_{j=1}^m$ is a non-increasing sequence of positive integers and $\hat A_0=[\hat a_{i,j}^{(0)}]_{1\le i,j\le m}$ is a real, symmetric and positive-definite matrix.
	\item If $u_m(0)\neq0$ (that is, if $k_p=0$), then
	\begin{align}
	\label{eq:asym2-1} a_{i,j}(\eps)&=\tilde a^{(0)}_{i,j}\,\eps^{\ell_i+\ell_j}+o\left(\eps^{\ell_i+\ell_j}\right) \quad\text{ for all }1\le i,j\le m-1,\\
	\label{eq:asym2-2}	a_{m,j}(\eps)&=a_{j,m}(\eps)=o\left(\frac{\eps^{\ell_j}}{|\log(\eps)|^{1/2}}\right) \quad\text{ for all }1\le j\le m-1,\\
	\label{eq:asym2-3} a_{m,m}(\eps)&=\frac{c}{|\log(\eps)|}+o\left(\frac{1}{|\log(\eps)|}\right)
	\end{align}
where $\{\ell_j\}_{j=1}^m$ is a non-increasing sequence of integers (with $\ell_m=0$ and $\ell_{m-1}>0$), the matrix $\tilde A_0=[\tilde a_{i,j}^{(0)}]_{1\le i,j\le m-1}$ is real, symmetric and positive-definite, and $c$ is a positive constant (explicitly: $c=2\pi\,u_m^2(0)$)
\item In both of the previous cases,
\begin{equation}
	\label{eq:asym3}
	b_{i,j}(\eps)=\delta_{i,j}+o(1)\quad \text{for all } 1\le i,j\le m. 
\end{equation}
\end{enumerate}  
\end{prop}

\begin{proof} The estimate \eqref{eq:asym3} follows from \eqref{eq:asym-b} and from the fact that the basis $\{u_j\}_{j=1}^m$ is $L^2$-orthonormal. 

If $u_i(0)=u_j(0)=0$, $Q(V^{u_i}_\eps)\sim \langle \Phi^{u_i},\Phi^{u_i}\rangle_{\mathcal X}\,\eps^{2\ell_i}$ and $Q(V^{u_j}_\eps)\sim \langle \Phi^{u_j},\Phi^{u_j}\rangle_{\mathcal X}\,\eps^{2\ell_j}$ according to \eqref{eq:cap-uv}. Then, Estimate \eqref{eq:asym-a} and Lemma \ref{l:V}, together with  \eqref{eq:cap-uv}, imply \eqref{eq:asym1} or \eqref{eq:asym2-1}, according to the case $u_m(0)=0$ or $u_m(0)\neq0$. 

If $u_m(0)\neq 0$, then \eqref{eq:cap-non-zero} implies
\[Q(V^{u_m}_\eps)\sim\frac{2\pi\,u_m^2(0)}{|\log(\eps)|},\]
yielding \eqref{eq:asym2-3} when combined with \eqref{eq:asym-a} and Lemma \ref{l:V}. For $1\le i \le m-1$, \eqref{eq:asym2-3} and Lemma \ref{l:V}, together with the estimates \eqref{eq:asym-a} and \eqref{eq:cap-mixed}, give
\[a_\eps(u_j,u_m)=O\left(\frac{\eps^{\ell_j}}{|\log(\eps)|}\right)+o\left(\frac{\eps^{\ell_j}}{|\log(\eps)|^{1/2}}\right)=o\left(\frac{\eps^{\ell_j}}{|\log(\eps)|^{1/2}}\right),\]
which is \eqref{eq:asym2-2}.
\end{proof}
\section{Asymptotic estimates of the eigenvalues}
\label{sec:ev}
\subsection{Statement of the results}

Let us stress that the analysis in this section is completely finite-dimensional. Indeed, we find asymptotic estimates for the eigenvalues of any problem defined using a pair of $m$-by-$m$ matrices $(A(\eps),B(\eps))$ satisfying the assumptions expressed either by \eqref{eq:asym1} and \eqref{eq:asym3} or by \eqref{eq:asym2-1}--\eqref{eq:asym2-3} and \eqref{eq:asym3}. We therefore take care to formulate the statements and the proofs in a way that does not involve any other object or property.

  We use repeatedly the following lemma. 
\begin{lem} \label{l:matrix-GS} Given a real, symmetric and positive-definite $m$-by-$m$ matrix $A$, there exists a unique pair of real $m$-by-$m$ matrices $(D,L)$ such that 
\begin{enumerate}[(i)]
	\item \label{cond:GSdiag} $D$ is diagonal,
	\item \label{cond:GSlower} $L$ is lower-triangular with all diagonal coefficients equal to $1$,
	\item \label{cond:GSprod} $D=L^TAL$.
\end{enumerate} 
Moreover, all the diagonal coefficients of $D$ are strictly positive, and the mapping $A\mapsto (D,L)$ is continuous (for any matrix norm).
\end{lem}

\begin{proof}
We observe that, up to flipping the order of lines and columns, the lemma describes the Cholesky factorization of $A$ (or more precisely its LDL variant). The properties are then well known  (see for instance \cite[Lemma 12.1.6]{schatzman}). 
\end{proof}

\begin{rem} \label{r:GS} In practice, the columns of the matrix $L$ are the basis of $\mathbb R^m$ that we obtain by applying the Gram-Schmidt orthogonalization procedure (without normalization) to the standard basis $\{e_i\}_{i=1}^m$ (with $e_i:=(\delta_{i,j})_{j=1}^m)$, starting from $e_m$ and using the scalar product derived from the matrix $A$, i.e. defined by $(x,y)\mapsto x^TAy$. Accordingly, we call $(D,L)$ the \emph{Gram-Schmidt pair} associated with $A$.
\end{rem}

Let us now state the two main results of this section. We assume throughout that $A(\eps)=[a_{i,j}(\eps)]_{1\leq i,j\leq m}$ and $B(\eps)=[b_{i,j}(\eps)]_{1\le i,j\le m}$ are two $m$-by-$m$ real and symmetric matrices depending on the real parameter $\eps>0$ (for $\eps$ small enough) and that $B(\eps)$ satisfies \eqref{eq:asym3}. In particular, this implies that, for $\eps>0$ small enough, $B(\eps)$ is positive-definite and the eigenvalue problem
\begin{equation}\label{eq:probl-old}
A(\eps)\xi = \mu\, B(\eps)\xi, \quad \xi\in \mathbb R^m,   
\end{equation}
has $m$ real eigenvalues, which we denote by $\{\mu_j(\eps)\}_{j=1}^m$ (counted with multiplicities and arranged in non-decreasing order). 

\begin{rem}\label{rem:ev-shifts} Let us note that, in the case where $A(\eps)$ and $B(\eps)$ are the matrices defined by the equations \eqref{eq:def-AB}, it follows from Lemma \ref{lem:shifts-2} that the eigenvalues $\{\mu_j(\eps)\}_{j=1}^m$ of Problem \eqref{eq:probl-old} coincide with the spectral shifts $\{\nu_{N+j-1}(\eps)\}_{j=1}^m$.
\end{rem}

\begin{prop}\label{p:finite-dim}
Let us assume that $A(\eps)$ satisfies \eqref{eq:asym1}. Let us denote by $\{k_i\}_{i=1}^p$ and $\{m_i\}_{i=1}^p$, respectively, the values taken by the finite sequence $\{\ell_j\}_{j=1}^{m}$ and  the corresponding multiplicities (so that these sequences are connected by the relations \eqref{eq:seq-l}). Let us denote the Gram-Schmidt pair associated with $\hat A_0$ by $(\hat D_0,\hat L_0)$, and let us write  $\hat D_0=[d^{(0)}_{i,j}]_{1\le i,j\le m}$ and $\hat L_0=[l^{(0)}_{i,j}]_{1\le i,j\le m}$.  
 
 For each $1\le i\le p$, let us define the $m_i$-by-$m_i$ matrices
\begin{align*}	
	\hat D_i^{(0)}&=[d^{(0)}_{i,j}]_{m_1+\dots+m_{i-1}+1\le i,j\le m_1+\dots+m_{i}},\\
	\hat L_i^{(0)}&=[l^{(0)}_{i,j}]_{m_1+\dots+m_{i-1}+1\le i,j\le m_1+\dots+m_{i}},\\
	\hat Q_i^{(0)}&=\left(\hat L_i^{(0)}\right)^T\hat L_i^{(0)},
\end{align*}
and let us denote by $\{\hat \mu_{i,j}^{(0)}\}_{j=1}^{m_i}$ the eigenvalues of the problem
\begin{equation}
	\label{eq:ev-block}
	\hat D_i^{(0)}\xi=\mu\,\hat Q_i^{(0)},\quad \xi\in \mathbb R^{m_i},
\end{equation} 
counted with multiplicities and arranged in non-decreasing order.

Then, we have the asymptotic estimates
\begin{equation}\label{eq:linalg1}
\mu_{m_1+\dots+m_{i-1}+j}(\eps)=\hat \mu_{i,j}^{(0)}\,\eps^{2 k_i} + o(\eps^{2k_i}) \quad\text{for all }1\le i\le p\text{ and }1\le j\le m_i.
\end{equation}
\end{prop}

\begin{rem} The matrices $\hat D_i$ and $\hat L_i$ are constructed from $\hat D$ and $\hat L$ by extracting successive $m_i$-by-$m_i$ diagonal blocks. \end{rem}

\begin{prop}\label{p:finite-dim-non-zero}
 Let us assume that $A(\eps)$ satisfies \eqref{eq:asym2-1}, \eqref{eq:asym2-2} and \eqref{eq:asym2-3}. Let us denote by $\{k_i\}_{i=1}^p$ and $\{m_i\}_{i=1}^p$ the sequences associated with $\{\ell_j\}_{j=1}^m$, as in Proposition \ref{p:finite-dim}. Let us denote the Gram-Schmidt pair associated with $\tilde A_0$ by $(\tilde D_0,\tilde L_0)$. Let us define, for each $1\le i\le p-1$, the $m_i$-by-$m_i$ matrices $\tilde D_i^{(0)}$, $\tilde L_i^{(0)}$ and $\tilde Q_i^{(0)}$ as well as the eigenvalues $\{\tilde \mu_{i,j}^{(0)}\}_{j=1}^{m_i}$ of the problem
\begin{equation}
	\label{eq:ev-block-non-zero}
	\tilde D_i^{(0)}\xi=\mu\,\tilde Q_i^{(0)},\quad \xi\in \mathbb R^{m_i},
\end{equation} 
as we did for $\hat A^{(0)}$ in Proposition \ref{p:finite-dim}.

Then, we have the asymptotic estimates
\begin{align}
\label{eq:linalg1-non-zero-1} \mu_{m_1+\dots+m_{i-1}+j}(\eps)&=\tilde \mu_{i,j}^{(0)}\,\eps^{2 k_i} + o\left(\eps^{2k_i}\right) \quad\text{for all }1\le i\le p-1\text{ and }1\le j\le m_i,\\
\label{eq:linalg1-non-zero-2}
\mu_{m}(\eps)&=\frac{c}{|\log(\eps)|}+o\left(\frac{1}{|\log(\eps)|}\right).
\end{align}
\end{prop}

\subsection{Proof of Proposition \ref{p:finite-dim}}

Let us define the matrices 
\begin{align*}
	P(\eps)&:={\rm diag}\left(\eps^{-\ell_1},\dots,\eps^{-\ell_m}\right),\\
	\hat A(\eps)&:=P(\eps)^TA(\eps)P(\eps)=P(\eps) A(\eps)P(\eps).
\end{align*}
It follows from Assumption \eqref{eq:asym1} that $\hat A(\eps)\to \hat A_0$. Since $\hat A_0$ is, by hypothesis, positive-definite, this convergence implies that $\hat A(\eps)$, and therefore also $A(\eps)$, is positive definite for $\eps>0$ small enough. Let us denote by $(\hat D(\eps),\hat L(\eps))$ the Gram-Schmidt pair associated with $\hat A(\eps)$. By continuity (see Lemma \ref{l:matrix-GS}), $\hat D(\eps)\to \hat D_0$ and $\hat L(\eps)\to \hat L_0$.  

We now denote by $(D(\eps),L(\eps))$ the Gram-Schmidt pair associated with $A(\eps)$.  We have
\begin{align*}
P(\eps) D(\eps) P(\eps)&=P(\eps)L(\eps)^TA(\eps)L(\eps)P(\eps)\\
&=P(\eps)L(\eps)^T\,P(\eps)^{-1}\,P(\eps)A(\eps)P(\eps)\,P(\eps)^{-1}L(\eps)P(\eps)\\
&=\left(P(\eps)^{-1}L(\eps)P(\eps)\right)^T{\hat A(\eps)}\, P(\eps)^{-1}L(\eps)P(\eps).
\end{align*}
where $P(\eps)D(\eps)P(\eps)$ is diagonal and $P(\eps)^{-1}L(\eps)P(\eps)$ is lower-triangular with all diagonal coefficients equal to $1$. It follows from the uniqueness of the Gram-Schmidt pair that $\hat D(\eps)= P(\eps)L(\eps)P(\eps)$ and $\hat L(\eps)= P(\eps)^{-1}L(\eps)P(\eps)$. We finally obtain
\begin{align}
\label{eq:D} D(\eps)&=P(\eps)^{-1}\hat D(\eps) P(\eps)^{-1},\\
\label{eq:L} L(\eps)&=P(\eps)\hat L(\eps)P(\eps)^{-1}.
\end{align}

Let us use the notation $D(\eps)={\rm diag}\,(d_1(\eps),\dots,d_m(\eps))$ and $L(\eps)=[l_{i,j}(\eps)]_{1\le i,j\le m}$.  Equation \eqref{eq:D} and the limit $\hat D(\eps)\to \hat D_0$ imply  
\begin{equation}
\label{eq:asym-d}
	d_j(\eps)=d_{j,j}^{(0)}\,\eps^{2\ell_j}+o\left(\eps^{2\ell_j}\right)\quad \text{for all }1\le j\le m.
\end{equation} 
Since $l_{i,j}(\eps)=0$ whenever $j>i$ and since Equation \eqref{eq:L} is equivalent to the family of identities
\[l_{i,j}(\eps)=\eps^{\ell_j-\ell_i}\hat l_{i,j}(\eps)\quad\text{ for all }1\le i,j\le m,\]
the limit $\hat L(\eps)\to \hat L_0$ implies (using the fact that $\{\ell_j\}_{j=1}^m$ is non-increasing)
\begin{equation*}
l_{i,j}(\eps)\to
\begin{cases}
	\hat l_{i,j}^{(0)} &\mbox{ if }\ell_i=\ell_j,\\
	0 &\mbox{ if }\ell_i\neq\ell_j.
\end{cases}
\end{equation*}
This is equivalent to $L(\eps)\to L_0$, with $L_0$ the $m$-by-$m$ block-diagonal matrix
\[\left[
	\begin{array}{cccc}
	\hat L_1^{(0)}&0&\cdots&0\\
	0&\hat L_2^{(0)}&\cdots&0\\
	\vdots&\vdots&\ddots&\vdots\\
	0&0&\cdots&\hat L_p^{(0)}
	\end{array}
	\right].\]

We now observe that the eigenvalues of Problem \eqref{eq:probl-old} are the same as those of the problem
\begin{equation}
\label{eq:probl-new}
	D(\eps)\xi=\mu\, Q(\eps)\xi,\quad \xi\in \mathbb{R}^m,
\end{equation}
with $Q(\eps)=L(\eps)^TB(\eps)L(\eps)$. Indeed, let us define two symmetric bilinear forms on $\mathbb R^m$ by
\begin{equation}
\label{eq:formsRm}
	a_\eps(\xi,\zeta):=\xi^TA(\eps)\zeta\quad\text{and\quad}b_\eps(\xi,\zeta):=\xi^TB(\eps)\zeta
\end{equation} 
(we chose the notation to be consistent with Section \ref{sec:finite}). Then, the eigenvalue problem: 
\begin{equation}
\label{eq:probl-forms}
	\text{find }(\mu,\xi)\in \mathbb R\times\mathbb R^m\text{ such that}\quad	a_\eps(\xi,\zeta)=\mu\,b_\eps(\xi,\zeta)\quad\text{for all }\zeta\in\mathbb R^m
\end{equation}
has the matrix form \eqref{eq:probl-old} when written in the standard basis $\{e_j\}_{j=1}^m$ and the matrix form \eqref{eq:probl-new} when written in the basis formed by the columns of $L(\eps)$. All these problems therefore have the same eigenvalues.

The limit $L(\eps)\to L_0$ and Assumption \eqref{eq:asym3} imply that $Q(\eps)\to Q_0$, with $Q_0$ the block-diagonal matrix
\[\left( L_0\right)^T L_0=\left[
	\begin{array}{cccc}
	\hat Q_1^{(0)}&0&\cdots&0\\
	0&\hat Q_2^{(0)}&\cdots&0\\
	\vdots&\vdots&\ddots&\vdots\\
	0&0&\cdots&\hat Q_p^{(0)}
	\end{array}
	\right].\]
 
Let us now define the block-diagonal matrix
\[\check D(\eps)=\left[
	\begin{array}{cccc}
	 \eps^{2k_1}\hat D_1^{(0)}&0&\cdots&0\\
	0&\eps^{2k_2}\hat D_2^{(0)}&\cdots&0\\
	\vdots&\vdots&\ddots&\vdots\\
	0&0&\cdots&\eps^{2k_p}\hat D_p^{(0)}
	\end{array}
	\right]\]
and denote by $\{\check  \mu_{j}(\eps)\}_{j=1}^m$ the eigenvalues of the problem 
\begin{equation}\label{eq:probl-limit}
\check  D(\eps)\xi =\mu\, Q_0\xi,\quad \xi\in\mathbb R^m,
\end{equation}
arranged in non-decreasing order and counted with multiplicities.

It follows from the block-diagonal structure of $\check D(\eps)$ and $Q_0$ that the values in the sequence $\{\check \mu_{j}(\eps)\}_{j=1}^m$ (counted with multiplicities) coincide with the values in the sequence
\[\left\{\hat \mu_{i,j}^{(0)}\,\eps^{2k_i}\,:\,1\le i\le p\mbox{ and }1\le j\le m_i\right\}.\]
Moreover, for $\eps$ small enough, the values in the second sequence are non-decreasing with respect to the lexicographic ordering of the indices $(i,j)$.

The estimates \eqref{eq:asym-d} and the limit $Q(\eps)\to Q_0$ imply the existence of positive functions $\eta(\eps),\, \theta(\eps)$ tending to $0$ such that 
\begin{align*}
&(1-\eta(\eps))\check D(\eps)\le D(\eps)\le (1+\eta(\eps))\check D(\eps)\\
&(1-\theta(\eps))Q_0\le Q(\eps)\le (1+\theta(\eps))Q_0
\end{align*}
(in the sense of the ordering of symmetric matrices). Thus, we can compare the Rayleigh quotients for $\eps>0$ small enough:
\[\frac{1-\eta(\eps)}{1+\theta(\eps)}\,\frac{\xi^T\,\check D(\eps)\,\xi}{\xi^T\,Q_0\,\xi}
\le\frac{\xi^T\,D(\eps)\,\xi}{\xi^T\,Q(\eps)\,\xi}\le\frac{1+\eta(\eps)}{1-\theta(\eps)}\frac{\xi^T\,\check D(\eps)\,\xi}{\xi^T\,Q_0\,\xi} \quad  \text{for all }\xi\in\mathbb R^m\setminus\{0\}.
\]
Using the minmax principle, we deduce that the eigenvalues of Problem \eqref{eq:probl-new} are asymptotically equivalent to the eigenvalues of Problem \eqref{eq:probl-limit}. This implies the asymptotic estimates \eqref{eq:linalg1}.
\subsection{Proof of Proposition \ref{p:finite-dim-non-zero}}

We perform the first step of a Gram-Schmidt orthogonalization with respect to the symmetric bilinear form $a_\eps(\cdot\,,\cdot)$ defined in \eqref{eq:formsRm} (which is positive-definite for $\eps>0$ small enough), starting from $e_m$. More explicitly, we define
\begin{align}\label{eq:secondbranch}
    & f_m^\eps :=e_m,\notag\\
    & f_{j}^\eps := e_{j}- \alpha_\eps^{(m,j)}e_m  \quad \text{for }1\le j\le m-1,
\end{align}
where (using Assumptions \eqref{eq:asym2-2} and \eqref{eq:asym2-3})
\begin{equation}\label{eq:alpha}
\alpha_\eps^{(m,j)}=\dfrac{a_\eps(e_m,e_j)}{a_\eps(e_m)}=\frac{a_{m,j}(\eps)}{a_{m,m}(\eps)}=o\left(\eps^{\ell_j}|\log(\eps)|^{1/2}\right)=o(1).
\end{equation}

In the new basis $\{f_i^\eps\}_{i=1}^m$,
we obtain, using Assumptions \eqref{eq:asym2-1}--\eqref{eq:asym2-3} and the estimate \eqref{eq:alpha},
\begin{align}
a_\eps( f_m^\eps) &= a_\eps(e_m,e_m)=\dfrac{c}{|\log\eps|} + o\left( \dfrac{1}{|\log\eps|} \right), \\
a_\eps(f_m^\eps,f_j^\eps) &=a_\eps(f_j^\eps,f_m^\eps)=0  \quad \text{for all  } 1\le j\le m-1,\\
 a_\eps( f_{i}^\eps,f_j^\eps) &= a_\eps(e_{i},e_j) - \alpha_\eps^{(m,j)} a_\eps(e_{i},e_m) - \alpha_\eps^{(m,i)}a_\eps(e_m,e_j)+\alpha_\eps^{(m,j)}\alpha_\eps^{(m,i)}a_\eps(e_m,e_m)\notag\\
 	&=
    \tilde a^{(0)}_{i,j}\,\eps^{\ell_i+\ell_j} + o\left(\eps^{\ell_i+\ell_j}\right) \quad \text{for all }1\le i,j\le m-1. \label{eq:asym-GS}
\end{align}
In addition, since $\alpha_\eps^{(m,j)}=o(1)$,
\begin{equation*}
b_\eps(f_i^\eps,f_j^\eps)=\delta_{i,j}+o(1)\quad\text{for all }1\le i,j \le m.
\end{equation*}
Let us denote by $A_1(\eps)$ and $B_1(\eps)$ the matrices  of $a_\eps(\cdot\,,\cdot)$ and $b_\eps(\cdot\,,\cdot)$, respectively, in the basis $\{f_i^\eps\}_{i=1}^m$. Then,  Problem \eqref{eq:probl-old} has the same eigenvalues as the problem 
\begin{equation}
	\label{eq:matrix-form-nonzero}
    A_1(\eps)\,\xi=\mu\,B_1(\eps)\xi,\quad\xi\in\mathbb R^m.
\end{equation}
The above asymptotic estimates imply that $B_1(\eps)\to I_m$ as $\eps\to0$ and that 
\begin{equation*}
    A_1(\eps)=
    \left[
    \begin{array}{rc}
        \tilde A(\eps)&\begin{array}{c}0\\ \vdots \\ 0\end{array}\\
        0\,\cdots\, 0&\frac{c}{|\log(\eps)|}+o\left(\frac{1}{|\log(\eps)|}\right)
    \end{array}
    \right]
\end{equation*}
with $\tilde A_1(\eps)$ an $(m-1)$-by-$(m-1)$ symmetric matrix. Using the notation $\tilde A(\eps)=[\tilde a_{i,j}(\eps)]_{i,j}$, we have, according to \eqref{eq:asym-GS}, \begin{equation*}
	\tilde a_{i,j}(\eps)=\tilde a^{(0)}_{i,j}\,\eps^{\ell_i+\ell_j}+o\left(\eps^{\ell_i+\ell_j}\right) \quad\text{ for all }1\le i,j\le m-1.
\end{equation*}
 We now apply Proposition \ref{p:finite-dim}, replacing $m$, $A(\eps)$ and $B(\eps)$ in the statement with $m-1$, $\tilde A(\eps)$ and $I_{m-1}$ (the $(m-1)$-by-$(m-1)$ identity matrix), respectively.  We note that the hypotheses of the  proposition are satisfied if we replace  $\{\ell_j\}_{j=1}^{m}$ and $\hat A_0$ in the statement with $\{\ell_j\}_{j=1}^{m-1}$ and $\tilde A_0$. Thus, the eigenvalues $\{\tilde \mu_j(\eps)\}_{j=1}^{m-1}$ of the problem
\begin{equation*}
	\tilde A(\eps)\xi=\mu \,\xi,\quad\xi\in \mathbb R^{m-1}
\end{equation*} 
satisfy
\begin{equation*}	
	\tilde \mu_j(\eps)= \tilde\mu_{i,j}^{(0)}\,\eps^{2\ell_j}+o\left(\eps^{2\ell_j}\right)\quad\text{for any }j=1,\dots,m-1.
\end{equation*}

Estimating the Rayleigh quotient $\xi^T A_1(\eps)\xi/\xi^T B_1(\eps)\xi$, similarly to the end of the proof of Proposition \ref{p:finite-dim}, we obtain
\begin{align*}
\mu_j(\eps)&\sim\tilde \mu_j(\eps)\quad\text{for all }1\le j\le m-1,\\
\mu_m(\eps)&\sim\frac{c}{|\log(\eps)|},
\end{align*}
which implies the asymptotic estimates \eqref{eq:linalg1-non-zero-1} and \eqref{eq:linalg1-non-zero-2}. 

\section{Conclusion and comments}

\subsection{Proof of Theorem \ref{thm:main} and error in Reference \cite{JFA2022}}

We can now complete the proof of Theorem \ref{thm:main}. In case \eqref{thm:case-i}, the first case of Proposition \ref{p:entries} implies that \eqref{eq:asym1} and \eqref{eq:asym3} hold. Then, Lemma \ref{lem:shifts-2} combined with Proposition \ref{p:finite-dim} yields the estimates \eqref{eq:case-i}, where the $\{\hat\nu_j\}_{j=1}^m$ are the terms of the sequence $\{\hat\mu^{(0)}_{i,j}\,;\, 1\le i\le p\text{ and }1\le j\le m_i\}$ arranged according to the lexicographic ordering of the indices $(i,j)$. In case \eqref{thm:case-ii}, the second case of Proposition \ref{p:entries} implies that the estimates \eqref{eq:asym2-1}--\eqref{eq:asym3} hold. Then, Lemma \ref{lem:shifts-2} combined with Proposition \ref{p:finite-dim-non-zero} yields the estimates \eqref{eq:case-ii-1} and \eqref{eq:case-ii-2}, where the $\{\hat\nu_j\}_{j=1}^{m-1}$ are the terms of the sequence $\{\tilde\mu^{(0)}_{i,j}\,;\, 1\le i\le p-1\text{ and }1\le j\le m_i\}$ arranged according to the lexicographic ordering of the indices $(i,j)$ and where $c=2\pi\,u_m^2(0)$. 

Let us note that, in the two dimensional setting of this note, the eigenvalue problems \eqref{eq:ev-block} and \eqref{eq:ev-block-non-zero} have size at most $2$-by-$2$. Thus, there are closed-form expressions for the eigenvalues $\{\hat \mu_{i,j}^{(0)}\}$ involving the entries of $\hat A_0$. Indeed, since the matrices of $\hat D_0$ and $\hat L_0$ are obtained from  $\hat A_0$ using a direct orthogonalization procedure (see Remark \ref{r:GS}), their entries are rational functions of the entries of $\hat A_0$, and the eigenvalues are at worst the roots of quadratic polynomials depending rationally on those entries. 

Let us consider in more detail the case where $m=2$, $p=2$ and $k_2>0$, so that $k_1=\ell_1>\ell_2=k_2>0$. We are then in the first case of Proposition \ref{p:entries}. To simplify notation, let us write the matrix appearing there as
\begin{equation}
\label{eq:ex-A0}
	\hat A_0=
	\left[
	\begin{array}{cc}
		\gamma_1 & \beta\\
		\beta   & \gamma_2\\
	\end{array}
	\right].
\end{equation}
As pointed out above, we obtain the associated Gram-Schmidt pair $(\hat D_0,\hat L_0)$ by applying an orthonormalization procedure (with respect to the scalar product derived from $\hat A_0$), which gives us
\begin{equation}\label{eq:ex-DL0}
 	\hat D_0=
	\left[
	\begin{array}{cc}
		\gamma_1-\frac{\beta^2}{\gamma_2} & 0\\
		0   & \gamma_2\\
	\end{array}
	\right]
	\quad\text{and}\quad
	\hat L_0=
	\left[
	\begin{array}{cc}
		1 & 0\\
		-\frac{\beta}{\gamma_2}   & 1\\
	\end{array}
	\right].
\end{equation}
Then, the matrices $\hat D_1^{(0)}$, $\hat D_2^{(0)}$, $\hat Q_1^{(0)}$ and $\hat Q_2^{(0)}$ appearing in Proposition \eqref{p:finite-dim} are all $1$-by-$1$, with
\begin{equation*}
	D_1^{(0)}=	\left[\gamma_1-\frac{\beta^2}{\gamma_2}\right],\ D_2^{(0)}=\left[\gamma_2\right]\ \text{and}\  Q_1^{(0)}=Q_2^{(0)}=\left[1\right]
\end{equation*} 
so that $\hat\mu_{1,1}^{(0)}=\gamma_1-\frac{\beta^2}{\gamma_2}$ and $\hat\mu_{2,1}^{(0)}=\gamma_2$. We conclude that 
\begin{equation*}
	\lambda_N^\eps=\lambda_N+\left(\gamma_1-\frac{\beta^2}{\gamma_2}\right)\,\eps^{2k_1}+o\left(\eps^{2k_1}\right)\quad\text{and}\quad\lambda_{N+1}^\eps=\lambda_N+\gamma_2\,\eps^{2k_2}+o\left(\eps^{2k_2}\right).
\end{equation*}

On the other hand, the application of Theorem 1.17 as stated in \cite{JFA2022} gives
\begin{equation*}
	\lambda_N^\eps=\lambda_N+\gamma_1\,\eps^{2k_1}+o\left(\eps^{2k_1}\right)\quad\text{and}\quad\lambda_{N+1}^\eps=\lambda_N+\gamma_2\,\eps^{2k_2}+o\left(\eps^{2k_2}\right)
\end{equation*}
evidencing the error in \cite{JFA2022}.

\subsection{Comments on the higher dimensional case in Reference \cite{NA2024}}

The corrections needed in \cite{JFA2022} have  consequences for some of the results in Reference \cite{NA2024}, which deals with higher dimensional cases. There, the eigenvalues (perturbed and unperturbed) are defined, as in Section \ref{sec:intro}, for sets $\Omega$ and $\omega$ in $\mathbb R^d$ (with $d\ge 3$) satisfying the properties in \cite[Definition 1.1]{NA2024} that is, open, bounded, connected, belonging to the class $C^{1,\alpha}$ for some $\alpha>0$, containing $0$ and having a connected complement. The case where the unperturbed eigenvalue $\lambda_N$ has multiplicity $m>1$ is studied in \cite[Section 7.2]{NA2024}. The main result there is Theorem 7.8, but its proof refers to the faulty proof of \cite[Theorem 1.17]{JFA2022}. The coefficients $\{\mu_{j,\ell}\}$ in \cite[Section 7.2]{NA2024} are therefore also incorrect.

It is easy to fix the proof of \cite[Theorem 7.8]{NA2024} using the method of this note. We first observe that the results of Section \ref{sec:finite} also hold for $\mathbb R^d$, with $d\ge3$. So does the decomposition of the unperturbed eigenspace $E$ according to the order of vanishing, and the existence of an adapted basis (the argument following Lemma \ref{l:basis} does not depend on the value of $d$). The asymptotic estimates of Proposition \ref{p:cap-estimate} can be replaced with 
\begin{equation*}
    Q(V^u_\eps,V^v_{\eps})=\eps^{\kappa(u)+\kappa(v)+d-2}\int_{\mathbb R^2}\nabla \Phi^u\cdot\nabla \Phi^v\,dx+o\left(\eps^{\kappa(u)+\kappa(v)+d-2}\right),
\end{equation*}
where, given $u\in E$, $\Phi^u$ is defined as the unique function harmonic in $\mathbb R^d\setminus\overline \omega$, equal to $u_\sharp$ on $\overline \omega$ and vanishing at $\infty$ (see \cite[Theorem 6.1 and Remark 6.2]{NA2024}). The matrices $A(\eps)$ and $B(\eps)$, defined as before by Equations \eqref{eq:def-AB}, therefore satisfy
\begin{align*}
    a_{i,j}(\eps)&=\hat a^{(0)}_{i,j}\,\eps^{\ell_i+\ell_j+d-2}+o\left(\eps^{\ell_i+\ell_j+d-2}\right), \\
    b_{i,j}(\eps)&=\delta_{i,j}+o(1)
\end{align*}
for all $1\le i,j\le m$, with $\hat A_0=[\hat a_{i,j}]_{1\le i,j\le m}$ defined by Equations \eqref{eq:G}.

Finally, we note that Proposition \ref{p:finite-dim} also holds, with an identical proof, when the values of the sequence $\{\ell_j\}_{1\le j\le m}$ are positive real numbers, not necessarily integers. We can therefore apply it after substituting $\{\ell_j+d/2-1\}_{1\le j\le m}$ for $\{\ell_j\}_{1\le j\le m}$ in the statement, and we obtain
\begin{equation}\label{eq:ev-d}
\lambda_{N+m_1+\dots+m_{i-1}+j}^\eps-\lambda_N=\hat\mu^{(0)}_{i,j}\eps^{2k_i+d-2}+o\left(\eps^{2k_i+d-2}\right)
\quad\text{for all }1\le i\le p\text{ and }1\le j\le m_i,
\end{equation}
where the $\{\hat\mu^{(0)}_{i,j}\}$ are defined from $\{m_i\}_{1\le i\le p}$ and $\hat A_0$ as in Proposition \ref{p:finite-dim}. 

We conclude by observing that the previous analysis shows Theorem 1.8 in \cite{NA2024} to be false in general. Indeed (still for $d\ge3$), let us consider an eigenvalue $\lambda_N$ with multiplicity $m=2$ such that $p=2$, so that $\ell_1=k_1>k_2=\ell_2\ge 0$. Let us denote by $\{u_1,u_2\}$ an adapted basis, and let us use, as before, the notation \eqref{eq:ex-A0} for $\hat A_0$. According to the estimates \eqref{eq:ev-d}, 
\begin{equation*}
	\lambda_N^\eps=\lambda_N+\left(\gamma_1-\frac{\beta^2}{\gamma_2}\right)\,\eps^{2k_1+d-2}+o\left(\eps^{2k_1+d-2}\right)\quad\text{and}\quad\lambda_{N+1}^\eps=\lambda_N+\gamma_2\,\eps^{2k_2+d-2}+o\left(\eps^{2k_2+d-2}\right).
\end{equation*}
Let us now assume that \cite[Theorem 1.8]{NA2024} is correct. Then,  there exists an orthonormal basis $\{v_1,v_2\}$ of $E$ such that 
\begin{align*}
	\lambda_N^\eps&=\lambda_N+\eps^{2\kappa(v_1)+d-2}\,\int_{\mathbb R^d}\left|\nabla \Phi^{v_1}\right|^2\,dx+o\left(\eps^{2\kappa(v_1)+d-2}\right),\\
    \lambda_{N+1}^\eps&=\lambda_N+\eps^{2\kappa(v_2)+d-2}\,\int_{\mathbb R^d}\left|\nabla \Phi^{v_2}\right|^2\,dx+o\left(\eps^{2\kappa(v_2)+d-2}\right).
\end{align*}
In particular, this implies that $\kappa(v_1)=k_1$, and therefore that $v_1$ is proportional to $u_1$. Since, by assumption, $v_2$ is orthogonal to $v_1$ and $u_1$ is orthogonal to $u_2$, $v_2$ is proportional to $u_2$, and since all these functions have unit $L^2$-norm , $\{v_1,v_2\}=\{\pm u_1,\pm u_2\}$. It follows that 
\begin{equation*}
	\lambda_N^\eps=\lambda_N+ \gamma_1\,\eps^{2k_1+d-2}+o\left(\eps^{2k_1+d-2}\right)\quad\text{and}\quad\lambda_{N+1}^\eps=\lambda_N+\gamma_2\,\eps^{2k_2+d-2}+o\left(\eps^{2k_2+d-2}\right),
\end{equation*}
which contradicts the previous estimates as soon as $\beta\neq 0$.

\bibliographystyle{acm}
\bibliography{biblio.bib}

@article {JFA2022,
    AUTHOR = {Abatangelo, Laura and L\'ena, Corentin and Musolino, Paolo},
     TITLE = {Ramification of multiple eigenvalues for the
              {D}irichlet-{L}aplacian in perforated domains},
   JOURNAL = {J. Funct. Anal.},
  FJOURNAL = {Journal of Functional Analysis},
    VOLUME = {283},
      YEAR = {2022},
    NUMBER = {12},
     PAGES = {Paper No. 109718, 50},
      ISSN = {0022-1236,1096-0783},
   MRCLASS = {35P20 (31B10 31C15 35B25 35C20)},
  MRNUMBER = {4489278},
       DOI = {10.1016/j.jfa.2022.109718},
       URL = {https://doi.org/10.1016/j.jfa.2022.109718},
}

@article {NA2024,
    AUTHOR = {Abatangelo, Laura and L\'ena, Corentin and Musolino, Paolo},
     TITLE = {Asymptotic behavior of generalized capacities with
              applications to eigenvalue perturbations: the higher
              dimensional case},
   JOURNAL = {Nonlinear Anal.},
  FJOURNAL = {Nonlinear Analysis. Theory, Methods \& Applications. An
              International Multidisciplinary Journal},
    VOLUME = {238},
      YEAR = {2024},
     PAGES = {Paper No. 113391, 34},
      ISSN = {0362-546X,1873-5215},
   MRCLASS = {35P15 (31B10 31C15 35B25 35C20)},
  MRNUMBER = {4648492},
       DOI = {10.1016/j.na.2023.113391},
       URL = {https://doi.org/10.1016/j.na.2023.113391},
}

@article {JST2019,
    AUTHOR = {Abatangelo, Laura and Felli, Veronica and Hillairet, Luc and
              L\'ena, Corentin},
     TITLE = {Spectral stability under removal of small capacity sets and
              applications to {A}haronov-{B}ohm operators},
   JOURNAL = {J. Spectr. Theory},
  FJOURNAL = {Journal of Spectral Theory},
    VOLUME = {9},
      YEAR = {2019},
    NUMBER = {2},
     PAGES = {379--427},
      ISSN = {1664-039X,1664-0403},
   MRCLASS = {35P20 (31C15 35J10 35P15)},
  MRNUMBER = {3950657},
MRREVIEWER = {Rodica\ Luca},
       DOI = {10.4171/JST/251},
       URL = {https://doi.org/10.4171/JST/251},
}

@article {ColboisColin1987,
    AUTHOR = {Colbois, B. and Colin de Verdi\`ere, Y.},
     TITLE = {Sur la multiplicit\'e{} de la premi\`ere valeur propre d'une
              surface de {R}iemann \`a{} courbure constante},
   JOURNAL = {Comment. Math. Helv.},
  FJOURNAL = {Commentarii Mathematici Helvetici},
    VOLUME = {63},
      YEAR = {1988},
    NUMBER = {2},
     PAGES = {194--208},
      ISSN = {0010-2571,1420-8946},
   MRCLASS = {58G25 (11F72 30F99)},
  MRNUMBER = {948777},
MRREVIEWER = {B.\ Helffer},
       DOI = {10.1007/BF02566762},
       URL = {https://doi.org/10.1007/BF02566762},
}

@article {Colin1986,
    AUTHOR = {Colin de Verdi\`ere, Yves},
     TITLE = {Sur la multiplicit\'e{} de la premi\`ere valeur propre non
              nulle du laplacien},
   JOURNAL = {Comment. Math. Helv.},
  FJOURNAL = {Commentarii Mathematici Helvetici},
    VOLUME = {61},
      YEAR = {1986},
    NUMBER = {2},
     PAGES = {254--270},
      ISSN = {0010-2571,1420-8946},
   MRCLASS = {58G25},
  MRNUMBER = {856089},
MRREVIEWER = {G\'erard\ Besson},
       DOI = {10.1007/BF02621914},
       URL = {https://doi.org/10.1007/BF02621914},
}

@article {Courtois1995,
    AUTHOR = {Courtois, Gilles},
     TITLE = {Spectrum of manifolds with holes},
   JOURNAL = {J. Funct. Anal.},
  FJOURNAL = {Journal of Functional Analysis},
    VOLUME = {134},
      YEAR = {1995},
    NUMBER = {1},
     PAGES = {194--221},
      ISSN = {0022-1236,1096-0783},
   MRCLASS = {58G25 (35P15)},
  MRNUMBER = {1359926},
MRREVIEWER = {Stig\ I.\ Andersson},
       DOI = {10.1006/jfan.1995.1142},
       URL = {https://doi.org/10.1006/jfan.1995.1142},
}

@book{schatzman,
 author = {Schatzman, Michelle},
 title = {Numerical analysis. {A} mathematical introduction. {Transl}. from the {French} by {John} {Taylor}},
 isbn = {0-19-850279-6; 0-19-850852-2},
 year = {2002},
 publisher = {Oxford: Clarendon Press},
 language = {English},
 zbMATH = {1667442},
 Zbl = {1019.65003}
}

@book{Fo95,
	Author = {Folland, Gerald B.},
	Edition = {Second},
	Isbn = {0-691-04361-2},
	Mrclass = {35-01},
	Mrnumber = {1357411},
	Pages = {xii+324},
	Publisher = {Princeton University Press, Princeton, NJ},
	Title = {Introduction to partial differential equations},
	Url = {https://mathscinet.ams.org/mathscinet-getitem?mr=1357411},
	Year = {1995}}

@article{DaMuRo15,
	Author = {Dalla Riva, M. and Musolino, P. and Rogosin, S. V.},
	Fjournal = {Asymptotic Analysis},
	Issn = {0921-7134},
	Journal = {Asymptot. Anal.},
	Mrclass = {35J25 (35B25 35C20)},
	Mrnumber = {3371119},
	Mrreviewer = {Jacek Banasiak},
	Number = {3-4},
	Pages = {339--361},
	Title = {Series expansions for the solution of the {D}irichlet problem in a planar domain with a small hole},
	Url = {https://mathscinet.ams.org/mathscinet-getitem?mr=3371119},
	Volume = {92},
	Year = {2015}}

\end{document}